\documentclass[a4paper,10pt]{amsart}
\usepackage[arrow,matrix]{xy}
\usepackage{color}
\usepackage{amsmath,amssymb,amscd,bbm,amsthm,mathrsfs,dsfont,enumerate}
\theoremstyle{plain} \textwidth=36pc \textheight=51pc

\newtheorem{theorem}{Theorem}[section]
\newtheorem{lemma}[theorem]{Lemma}
\newtheorem{example}[theorem]{Example}
\newtheorem{proposition}[theorem]{Proposition}

\theoremstyle{definition}
\newtheorem{definition}[theorem]{Definition}
\newtheorem{remark}[theorem]{Remark}
\newtheorem{question}[theorem]{Question}
\numberwithin{equation}{section}

 \DeclareMathOperator{\Ext}{Ext}
\DeclareMathOperator{\Hom}{Hom} \DeclareMathOperator{\Tor}{Tor}

\DeclareMathOperator{\gr }{gr}
\DeclareMathOperator{\de }{d\!}

\DeclareMathOperator{\Der }{Der}
\DeclareMathOperator{\HP }{HP}
\DeclareMathOperator{\tr }{tr}

\begin{document}
\title{A note on the duality between Poisson homology and cohomology}
\author{Jiafeng L\"u, Xingting Wang and Guangbin Zhuang}
\address{L\"u: Department of Mathematics, Zhejiang Normal University, Jinhua, China}
\email{jiafenglv@gmail.com}

\address{Wang: Department of Mathematics,
University of Washington, Seattle, Washington 98195, USA}

\email{xingting@uw.edu}

\address{Zhuang: Department of Mathematics,
University of Southern California, Los Angeles 90089-2532, USA}

\email{gzhuang@usc.edu}

\begin{abstract} 
For a Poisson algebra $A$, by studying its universal enveloping algebra $A^{pe}$, we prove a duality theorem between Poisson homology and cohomology of $A$.
\end{abstract}



\maketitle

\medskip
\section{introduction}
Poisson algebras, whose deformation are intimately related to some interesting non-commutative algebras (for example, Calabi-Yau algebras), have been intensively studied recently. It is known that for a lot of Poisson algebra $A$ of dimension $\ell$, the following duality holds
\begin{equation}\label{Poincare}
\HP_n(A)\cong \HP^{\ell-n}(A)
\end{equation}
where $\HP_*(A)$ and $\HP^*(A)$ are Poisson homology and Poisson cohomology of $A$, respectively. In this paper, we prove a criterion for the duality $(\ref{Poincare})$ to hold. In fact, for a large class of Poisson algebras $A$ of dimension $\ell$, we are able to construct a right Poisson $A$-module $\omega_A$ such that 
\begin{equation}\label{general}
\HP_n(A, \omega_A)\cong \HP^{\ell-n}(A)
\end{equation} 
for any $n$. See Proposition \ref{main} for the details. The duality $(\ref{general})$ has been observed by Launois-Richard for a class of quadratic Poisson algebras \cite[Section 3.1]{LR}. The proof of our result will use the universal enveloping algebra $A^{pe}$ of a Poisson algebra $A$. 

Throughout this paper, let $k$ denote a base field of characteristic $0$. All algebras and tensor products are taken over $k$ unless otherwise stated.

\section{Universal enveloping algebras of Poisson algebras}
For a Poisson algebra $A$, one can define its universal enveloping algebra $A^{pe}$ \cite{Oh}. In \cite{U}, a constructive definition in terms of generators and relations are given. It is also observed, by several authors \cite{LWZ, T}, that $A^{pe}$ is canonically isomorphic to $V(A, \Omega_A)$, where $\Omega_A$ is the K\"ahler differential of $A$ and is equipped with the Lie-Rinehart algebra structure derived from the Poisson structure of $A$. For a detailed account of this isomorphism, one can refer to \cite[Proposition 5.7]{LWZ}.

The construction of $A^{pe}$ results in an algebra map $m: A\rightarrow A^{pe}$ and a Lie map $h: A\rightarrow A^{pe}$. The map $m$ is injective and therefore we would simply consider $A$ as a subalgebra of $A^{pe}$. Also, we will write $h_a$ for $h(a)$.

For a left Poisson module $M$ with the bilinear map $\{, \}: A\otimes M\rightarrow M$, it naturally becomes left $A^{pe}$-module such that $\{a, m\}_M=h_am$ for any $a\in A$ and $m\in M$. A similar statement holds for right Poisson modules as well.
The following proposition is well known. See \cite[Corollary 1]{U}.
\begin{proposition}
The category of left (resp. right) Poisson modules over a Poisson algebra $A$ is equivalent to the left (resp. right) module category over $A^{pe}$.
\end{proposition}

Clearly, $A$ can be viewed as a left $A^{pe}$-module as well as a right $A^{pe}$-module. In fact, as a left (resp. right) $A^{pe}$-module, $A$ is isomorphic to the quotient of $A^{pe}$ by the left (resp. right) ideal generated by elements $h_a$ where $a\in A$.

The advantage of identifying $A^{pe}$ with $V(A, \Omega_A)$ is that now we can use some standard results from the theory of Lie-Rinehart algebras. For example, the algebra $V(A, \Omega_A)$ carries a filtration which naturally passes to $A^{pe}$ via the canonical isomorphism. The filtration $\{F_n\}_{n\ge 0}$ is such that $A\subset F_0 A^{pe}$ and $h_x\in F_1A^e$ for any $x\in A$. Now the theorem \cite[Theorem 3.1]{Ri} says the following.

\begin{proposition}\label{PBW}
Let $A$ be a Poisson algebra. If the K\"ahler differential $\Omega_A$ is a projective $A$-module, then there is an $A$-algebra isomorphism 
\begin{equation}
S_A(\Omega_A)\cong \gr_F A^{pe},
\end{equation}
where $S_A(\Omega_A)$ is the symmetric $A$-algebra on $\Omega_A$.
\end{proposition}

Here are some other consequences that will play essential roles in our paper. 

Let $A$ be a Poisson algebra. Consider the complex $C_*$ where $C_n=0$ for $n<0$ and $C_n=A^{pe}\otimes_A\Omega_{A/k}^n$ for $n\ge 0$. Conventionally, we take $\Omega_{A/k}^0=A$. The differential $b$ is given by
\begin{align*}
b(a_0\otimes \de a_1\de a_2\cdots \de a_{n})&=\sum_{i=1}^{n}(-1)^{i+1}a_0h_{a_i}\otimes \de a_1\de a_2\cdots\hat{\de a_i} \cdots\de a_{n}\\
&+ \sum_{1\le i<j\le n}(-1)^{i+j}a_0\otimes \de \{a_i, a_j\} \de a_1\cdots \hat{\de a_i} \cdots\hat{\de a_j} \cdots\de a_n.
\end{align*}

\begin{proposition}\label{resolution}
 Let $A$ be a Poisson algebra and suppose that $\Omega_A$ is projective over $A$. Then the complex $C_*$ defined above is a projective resolution of $A$ as a left $A^{pe}$-module.
\end{proposition}
\begin{proof}
This is \cite[Lemma 4.1]{Ri}.
\end{proof}
Therefore, we get the following proposition, which to our knowledge is first explicitly spelled out in \cite{H}. 
\begin{proposition}\label{Poissonext}
Let $A$ be a Poisson algebra and suppose that $\Omega_A$ is projective over $A$. Let $M$ be a left Poisson module and $N$ a right Poisson module. Then
\begin{equation}
\HP^*(A, M)\cong \Ext^*_{A^{pe}}(A, M),
\end{equation}
and 
\begin{equation}
\HP_*(A, N)\cong \Tor_*^{A^{pe}}(N, A),
\end{equation}
where $\HP^*(A, M)$ is the Poisson cohomology with coefficients in $M$ and $\HP_*(A, N)$ is the Poisson homology with coefficients in $N$.
\end{proposition}

In particular, if $M=A$ (resp. $N=A$), then $\HP^*(A, M)$ (resp. $\HP_*(A, N)$) is simply denoted by $\HP^*(A)$ (resp. $\HP_*(A))$ and called the Poisson cohomology (resp. Poisson homology) of $A$. However, in light of the previous proposition, one needs not to worry too much about those definitions, at least when $\Omega_A$ is projective over $A$, since they are just $\Ext$ and $\Tor$ which we assume most readers have a solid understanding on.

The rest of the section will be devoted to some details on the Lie-Rinehart algebra structure on $\Omega_A$. First, let's recall the definition.

\begin{definition}\label{LR}
Let $R$ be a commutative ring with identity, $A$ a commutative $R$-algebra and $L$ a Lie algebra over $R$. The pair $(A, L)$ is called a {\it Lie-Rinehart algebra} over $R$ if $L$ is a left $A$-module and there is an {\it anchor map} $\rho: L\rightarrow \Der_R(A)$, which is an $A$-module and a Lie algebra morphism, such that the following relation is satisfied,
\begin{equation}
[\xi, a\cdot\zeta]= a\cdot [\xi, \zeta]+ \rho(\xi)(a)\cdot \zeta,
\end{equation}
for any $a\in A$ and $\xi, \zeta\in L$.
\end{definition}

For simplicity, we will use $\xi(a)$ for $\rho(\xi)(a)$ where $\xi\in L$ and $a\in A$. In this paper, we would always assume that $R$ is the ground field $k$.  The following example is the one that we are mainly interested \cite[Theorem 3.8]{H}.
\begin{example}\label{differential}
{\rm Let $A$ be a Poisson algebra over $k$ and $\Omega_A$ its K\"ahler differentials. Then the pair $(A, \Omega_A)$ becomes a Lie-Rinehart algebra over $k$ where the Lie bracket on $\Omega_A$ is given by 
\begin{equation}
[a df, bdg]= abd\{f, g\}+ a\{f, b\}dg-b\{g, a\}df,
\end{equation}
and the anchor map sends $df$ to $\{f, \cdot \}$.}
\end{example}

Now suppose that $A$ is an affine Poisson algebra and that $\Omega_A$ is free over $A$ of rank $\ell$ (which is equal to the Krull dimension of $A$) with a basis $\{\de x_1, \cdots, \de x_\ell\}$. Then for any $y\in A$, there is a unique matrix $Y=(Y_{ij})\in M_\ell(A)$ such that 
$$[\de y, \de x_i]=\sum_j Y_{ij} \de x_j$$
for any $i$. Define the trace of $\de y$ to be
\begin{equation}\label{trace}
\tr (\de y)= \tr Y.
\end{equation}

\section{A duality between Poisson homology and cohomology}
In this section, we are going to prove the main result of the paper. 
\begin{lemma}\label{highestdegree}
Suppose that $A$ is an affine Poisson algebra and that $\Omega_A$ is free over $A$ of rank $\ell$. Then $\Ext^i_{A^{pe}}(A, A^{pe})=0$ for $i\neq \ell$ and $\Ext^{\ell}_{A^{pe}}(A, A^{pe})\neq 0$.
\end{lemma}
\begin{proof}
By Proposition \ref{PBW}, $\gr A^{pe}\cong A[y_1, \cdots, y_\ell]:= E$. Now we use the standard spectral sequence
$$ \Ext^{*}_{E}(A, E) \Rightarrow \Ext^{*}_{A^{pe}}(A, A^{pe}).$$
Now the result follows from the fact that $\Ext^i_E(A, E)=0$ for $i\neq \ell$ and $\Ext^\ell_E(A, E)\neq 0$.
\end{proof}

\begin{proposition}
Suppose that $A$ is an affine Poisson algebra and that $\Omega_A$ is free over $A$ of rank $\ell$. Let $\omega_A$ be the right $A^{pe}$-module $\Ext^{\ell}_{A^{pe}}(A, A^{pe})$.
Then for any left $A^{pe}$-module (or equivalently, left Poisson module over $A$) $N$,
\begin{equation}
\Tor^{A^{pe}}_n(\omega_A, N)\cong \Ext^{\ell-n}_{A^{pe}}(A, N).
\end{equation}
\begin{proof}
This is a consequence of Ischebeck's spectral sequence and Lemma \ref{highestdegree}. For the details, please see \cite[Theorem 1.1, Corollary 1.4]{K}.
\end{proof}
\end{proposition}

Next we look closer at $\omega_A=\Ext^{\ell}_{A^{pe}}(A, A^{pe})$.
\begin{lemma}\label{Nakauto}
Let $A$ be as in Lemma \ref{highestdegree} and let $\{\de x_1, \cdots, \de x_\ell\}$ be a free $A$-basis of $\Omega_A$. Then $\omega_A=\Ext^{\ell}_{A^{pe}}(A, A^{pe})$ is isomorphic to $A^{pe}/J$ as right $A^{pe}$-modules, where $J$ is the right ideal generated by elements $h_{x_i}- \tr(\de x_i)$, $i=1, 2, \cdots, \ell$. \end{lemma}
\begin{proof}
The proof is just a direct calculation. Let $C_*$ be the resolution as in Proposition \ref{resolution}. Then $\Ext^{\ell}_{A^{pe}}(A, A^{pe})$ is the cokernel of the map
\begin{equation}\label{partial}
\Hom_{A^{pe}}(C_{\ell-1}, A^{pe})\xrightarrow{\partial}\Hom_{A^{pe}}(C_{\ell}, A^{pe}).
\end{equation}
Notice that $C_{\ell}$ is a free left $A^{pe}$-module of rank $1$. In fact, given a free basis $\{\de x_1, \cdots, \de x_\ell\}$ of $\Omega_A$, $C_{\ell}$ is a free left $A^{pe}$-module with a basis $\{1\otimes \de x_1\de x_2....\de x_\ell\}$. Consequently, by sending $f\in \Hom_{A^{pe}}(C_{\ell}, A^{pe})$ to $f(1\otimes \de x_1\de x_2....\de x_\ell)\in A^{pe}$, we can identify $\Hom_{A^{pe}}(C_{\ell}, A^{pe})$ with $A^{pe}$ as right $A^{pe}$-modules. Now a direct calculation shows that the image of $\partial$ in $(\ref{partial})$ is the right ideal of $A^{pe}$ generated by elements of the form $h_{x_i}- \tr(\de x_i)$ where $i=1, 2,\cdots, \ell$. This completes the proof.
\end{proof}

\begin{question} Is there an automorphism $\nu$ on $A^{pe}$ such that $\nu$ restricts to identity on $A$ and $\omega_A\cong A^{\nu}$ as right $A^{pe}$-modules?
\end{question}

\begin{remark}\label{explicit}
Let everything be as in Lemma \ref{Nakauto}, then the canonical map $A\rightarrow A^{pe}/J=\omega_A$ induced by the algebra map $m: A\rightarrow A^{pe}$ is an isomorphism of right $A$-modules. This is an easy consequence of Proposition \ref{PBW}. Under this identification, the right $A^{pe}$-module structure on $\omega_A=A$ is given by 
\begin{equation}
mh_{x_i}= -\{x_i, m\}+m \tr(\de x_i),
\end{equation}
where $m\in A$. Or, if we think $\omega_A$ as a right Poisson module over $A$ with bilinear map $\{, \}_\omega: \omega_A\otimes A\rightarrow \omega_A$, the previous equation just translate into
\begin{equation}
\{m, x_i\}_\omega= -\{x_i, m\}+m \tr(\de x_i)
\end{equation}
This equation has been observed by Launois-Richard for a class of quadratic Poisson algebras \cite{LR}.
\end{remark}

Now we are ready to deliver the main result of this note. 
\begin{proposition}\label{main}
Retain the notation in Lemma \ref{highestdegree} and let $\{\de x_1, \cdots, \de x_\ell\}$ be a free $A$-basis of $\Omega_A$. Then 
\begin{equation}
\HP_n(A, \omega_A)\cong \HP^{\ell-n}(A)
\end{equation}
for any $n$.
Moreover, if $\tr(\de x_i)=0$ for any $i$, then 
\begin{equation}
\HP_n(A)\cong \HP^{\ell-n}(A)
\end{equation}
for any $n$.
\end{proposition}
\begin{proof}
The first statement is a consequence of Proposition \ref{Poissonext} and Proposition \ref{Nakauto}.
If $\tr(\de x_i)=0$ for any $i$, by Lemma \ref{Nakauto}, $\omega_A$ is isomorphic to $A^{pe}/J$ where $J$ is the right ideal generated by $h_{x_i}$. Hence $\omega_A$ is isomorphic to $A$ as right $A^{pe}$-modules and therefore we have the second statement.
\end{proof}

\section{Examples}

In this section, we look at some examples.
\subsection{A Poisson algebra arising from a Calabi-Yau algebra} Let $A=k[x, y, z]$ with Poisson bracket given by
\begin{equation}
\{z, y\}=2xz, \quad \{z, x\}=0, \quad \{y, x\}=x^2.
\end{equation}
This Poisson algebra is studied by Berger-Pichereau in \cite{BP}, whose deformation gives a type of Calabi-Yau algebra. In the same paper, the Poisson homology of $A$ is also explicitly calculated \cite[Proposition 5.7]{BP}. Clearly, the K\"alher differential $\Omega_A$ is free over $A$ with a basis $\{\de x, \de y, \de z\}$. Also, 
\begin{align*}
[\de z, \de y]&=\de\,\{z, y\}= 2x\de z+2z\de x,\\
 [\de z, \de x]&=\de \,\{z, x\}=0,\\
  [\de y, \de x]&=\de \,\{y, x\}=2x\de x.
\end{align*}
Consequently, $\tr(\de x)=\tr(\de y)=\tr(\de z)=0$ and therefore $\HP_n(A)\cong \HP^{3-n}(A)$. In fact, as pointed out in \cite{BP}, the Poisson algebra $A$ is derived from the Poisson potential $\phi=-x^2z$ and thus the duality is automatic.

\subsection{A class of quadratic Poisson algebras} In \cite{LR}, Launois-Richard studied the following Poisson algebra $A$. As an algebra, $A$ is $k[X_1, X_2, \cdots, X_\ell]$ and the Poisson bracket is given by
\begin{equation}
\{X_i, X_j\}=a_{ij}X_iX_j
\end{equation}
where $(a_{ij})\in M_n(k)$ is an antisymmetric matrix. A direct calculation shows that 
\begin{equation}
\tr(\de X_i)= (\sum_{j=1}^\ell a_{ij})X_i
\end{equation}
Therefore, as observed in Remark \ref{explicit}, $\omega_A$ is isomorphic to $A$ as right $A$-modules and the right Poisson module structure on $\omega_A=A$ is given by 
\begin{equation}\label{quadratic}
\{m, X_i\}_\omega= -\{X_i, m\}+(\sum_{j=1}^\ell a_{ij})mX_i
\end{equation}
for any $m\in \omega_A=A$. Notice that the equation $(\ref{quadratic})$ is exactly \cite[Section 3.1(6)]{LR}. Moreover, the twisted duality $\HP_n(A, \omega_A)\cong \HP^{\ell-n}(A)$ for this particular type of Poisson algebra is given in \cite[Theorem 3.4.2]{LR}.

In fact, for this example, we can find an algebra automorphism $\nu$ on $A^{pe}$ such that $\nu|_A=id$ and $\omega_A\cong A^{\nu}$. It is the unique map given by
$$\nu(X_i)=X_i, \quad \nu(h_{X_i})=h_{X_i}+(\sum_{j=1}^\ell a_{ij})X_i$$
for any $i$. 


\begin{thebibliography}{99}

\bibitem[BP]{BP} R. Berger, A. Pichereau, {\it Calabi-Yau algebras viewed as deformations of Poisson algebras}, Algebr. Represent. Theory (2013), http://dx.doi.org/10.1007/s10468-013-9417-z, in press.
\bibitem[H]{H} J. Huebschmann, {\it Poisson cohomology and quantization}, J. Reine Angew. Math. 408 (1990) 57-113.
 

\bibitem[K]{K} Ulrich Kr\"ahmer, {\it Poincar\'e duality in Hochschild (co)homology}, in: New Techniques in Hopf Algebras and Graded Ring Theory, K. Vlaam. Acad. Belgie Wet. Kunsten (KVAB), Brussels, 2007, pp. 117-125.


\bibitem[LR]{LR} St\'ephane Launois, Lionel Richard, {\it Twisted Poincar\'e duality for some quadratic Poisson algebras}, Lett. Math. Phys. 79 (2) (2007) 161-174.
\bibitem[LWZ]{LWZ} J.-F. L\"u, X. Wang, G. Zhuang, {\it Universal enveloping algebras of Poisson Hopf algebras}, arXiv:1402.2007, 2014.

\bibitem[Oh]{Oh} S.-Q. Oh, {\it Poisson enveloping algebras}, Comm. Algebra 27 (1999), 2181-2186.
\bibitem[Ri]{Ri} G. S. Rinehart, {\it Differential forms on general commutative algebras}, Trans. Amer. Math. Soc. 108 (1963), 195-222.


\bibitem[T]{T} Matthew Towers, {\it Poisson and Hochschild cohomology and the semiclassical limit}, arXiv:1304.6003, 2013.


\bibitem[U]{U} U. Umirbaev, {\it Universal enveloping algebras and universal derivations of Poisson algebras}, J. Algebra 354 (2012), 77-94.



\end{thebibliography}
\end{document}